\documentclass[english]{extarticle}
\usepackage[T1]{fontenc}
\usepackage[latin9]{inputenc}
\usepackage{geometry}
\usepackage{color}
\usepackage{babel}
\usepackage{float}
\usepackage{units}
\usepackage{mathrsfs}
\usepackage{amsmath}
\usepackage{amsthm}
\usepackage{amssymb}
\usepackage{graphicx}
\usepackage[numbers]{natbib}
\usepackage[unicode=true,pdfusetitle,
 bookmarks=true,bookmarksnumbered=false,bookmarksopen=false,
 breaklinks=false,pdfborder={0 0 1},backref=false,colorlinks=true]
 {hyperref}
\hypersetup{
 linkcolor=blue, urlcolor=marineblue, citecolor=blue, pdfstartview={FitH}, hyperfootnotes=false, unicode=true}

\makeatletter

\numberwithin{equation}{section}
\numberwithin{figure}{section}
\theoremstyle{plain}
\newtheorem{thm}{\protect\theoremname}
\theoremstyle{definition}
\newtheorem{defn}[thm]{\protect\definitionname}
\theoremstyle{plain}
\newtheorem{lem}[thm]{\protect\lemmaname}
\theoremstyle{remark}
\newtheorem{rem}[thm]{\protect\remarkname}
\theoremstyle{definition}
\newtheorem{example}[thm]{\protect\examplename}
\theoremstyle{plain}
\newtheorem{prop}[thm]{\protect\propositionname}

\usepackage{lettrine}
\let\myFoot\footnote
\renewcommand{\footnote}[1]{\myFoot{#1\vspace{0.8mm}}}
\usepackage{lmodern}
\usepackage[T1]{fontenc}
\usepackage{concmath}
\usepackage{ae,aecompl}
\usepackage[T1]{fontenc}
\usepackage{amsmath}
\usepackage{graphicx}

\usepackage{tikz}

\makeatother

\providecommand{\definitionname}{Definition}
\providecommand{\examplename}{Example}
\providecommand{\lemmaname}{Lemma}
\providecommand{\propositionname}{Proposition}
\providecommand{\remarkname}{Remark}
\providecommand{\theoremname}{Theorem}

\def \Om{\Omega}
\def \RR{{\mathbb R}}

\def \CC{{\mathbb C}}

\begin{document}
\title{On the Dirichlet eigenvalue problem and the conformal Skorokhod
embedding problem}

\author{Maher Boudabra, Greg Markowsky\\
Monash University}

\maketitle

\abstract{In a recent work by Gross, the following problem was stated and solved: given a measure $\mu$ with finite second moment, find a simply connected domain $U$ in $\CC$ such that the real part of a Brownian motion stopped when it leaves $U$ is distributed as $\mu$. The construction developed by Gross yields a domain which is symmetric with respect to the real axis, but it has been noted by other authors that other domains are also possible, in particular there are a number of examples which have the property that a vertical ray starting at a point in the domain lies entirely within the domain. In this paper we give a new solution to the problem posed by Gross, and show that these other cases noted before are special cases of this method. We further show that the domain generated by this method has the property that it always has the minimal rate (as defined in terms of the spectrum of the Laplacian operator) among all possible domains corresponding to a fixed distribution $\mu$, which gives a partial solution to a question posed by Mariano and Panzo. We show that the domain is unique, provided certain conditions are imposed, and use this to give several examples. We also describe a method for identifying the boundary curve of the domain, and discuss several other related topics.}

\section{Introduction}

In what follows, $Z_{t}$ is a standard planar Brownian motion starting at 0, and for any plane domain $U$ containing 0 we let $\tau_{U}$ denote the first exit time of $Z_t$ from $U$. In the recent paper \cite{boudabra2019remarks} the following theorem was proved, which is a direct generalization of the elegant results and methods developed in \citep{gross2019}.

\begin{thm} \label{newguy}
Given a nondegenerate probability distribution $\mu$ on $\RR$ with zero mean and finite nonzero $p$-th moment (with $1<p<\infty$),
we can find a simply connected domain $U$ such that $\Re(Z_{\tau_{U}})$ has
the distribution $\mu$. Furthermore we have $E[(\tau_{U})^{p/2}]<\infty$.
\end{thm}

Note that by {\it nondegenerate} we mean that $\mu$ is not a point mass at the origin. In fact, this qualification is not present in either \citep{gross2019} or \cite{boudabra2019remarks}, and it seems to have been overlooked. But it is necessary, and we will discuss this more later in the paper.

In what follows, when $\mu$ is given we will refer to $U$ as a $\mu${\it-domain}. Therefore Theorem \ref{newguy} provides the existence of a $\mu$-domain whenever $\mu$ satisfies the moment conditions. In general a $\mu$-domain is not unique, however a uniqueness principle for this construction with additional conditions was proved in \cite{boudabra2019remarks}. We will say that a domain $U$ is {\it symmetric} if $\bar z \in U$ whenever $z \in U$. We will call a domain $U$ {\it $\Delta$-convex} if, whenever $z_1, z_2 \in U$ with $\Re(z_1) = \Re(z_2)$ then the vertical line segment connecting $z_1$ and $z_2$ lies entirely within $U$. With these definitions, the uniqueness principle is as follows.

\begin{thm} \label{thm:The-distribution-}
For any nondegenerate distribution $\mu$ satisfying the conditions of Theorem \ref{newguy}, there is a unique domain $U$ such that $\Re(Z_{\tau_\Om}) \sim \mu$ and which is symmetric, $\Delta$-convex, and satisfies $E[(\tau_{\Om})^{p/2}] < \infty$.
\end{thm}

It was pointed out in \cite{boudabra2019remarks} that this result fails if any of the conditions is omitted, and in particular it was shown that the parabola and horizontal strip are $\mu$-domains when $\mu$ is the hyperbolic secant distribution. Another example of this phenomenon was demonstrated in \cite{mariano2020conformal}, where it was shown that the catenary (see Figure \ref{cat11} below) is a $\mu$ domain when $\mu$ is the uniform distribution on $(-1,1)$, even though it can not be the domain generated by Gross' construction as it is not symmetric. Furthermore, the authors of \cite{mariano2020conformal} showed that, among all $\mu$-domains for the uniform distribution, the catenary is the one with the minimal principle Dirichlet eigenvalue, and asked for a characterization of $\mu$-domains which are extremal with respect to the principle Dirichlet eigenvalue.

In this paper, our primary intention is to demonstrate a new method for solving the conformal Skorohod problem, one which gives the parabola and catenary when applied to the hyperbolic secant and uniform distributions, respectively. This method also has the property that its solution is always the one with the minimal principle Dirichlet eigenvalue among all $\mu$-domains, which therefore gives a partial solution to the problem posed by Mariano and Panzo in \cite{mariano2020conformal}.

To state our results we need a definition. We will say that a domain is $\Delta^\infty${\it -convex} if, given any $z \in U$, the vertical ray $\{w:\Re(w) = \Re(z), Im(w) \geq Im(z)\}$ lies entirely in $U$. So, for example, the parabola and catenary described above are $\Delta^\infty$-convex, while a horizontal strip is not. The reason for this name is that it is a variation on the notion of $\Delta$-convexity, as defined above. Our primary results are as follows.

\begin{thm}
\label{Existence1} If $\mu\in L^{p}$ is nondegenerate for some $p>1$ then there exists
a $\mu$-domain $U$ containing zero which is $\Delta^\infty$-convex. Furthermore $\mathbf{E}(\tau_{U}^{p/2})<\infty$.
\end{thm}

\begin{thm}
\label{uniqueness1}The domain $U$ given in Theorem \ref{Existence1} is the unique $\mu$-domain which is $\Delta^\infty$-convex and satisfies $\mathbf{E}(\tau_{U}^{p/2})<\infty$ for some $p>1$.
\end{thm}

\begin{thm}
\label{minimality of the domain1}The $\mu$-domain $U$ constructed
in Theorem \ref{Existence1} is always the one with the minimal principle Dirichlet eigenvalue among all $\mu$-domains.
\end{thm}

Sections \ref{prelimsec}, \ref{constructsec}, and \ref{uniquesec} are devoted to the proofs of these theorems. The domain generated by our method is bounded below by a boundary curve, and in Section \ref{boundarysec} we describe a method of determining this curve from the measure $\mu$. Finally, in Section \ref{cauchy}, we present a curiosity, that a formal application of our methods to the Cauchy distribution yields the correct $\Delta^\infty$-convex domain, even though the Cauchy does not satisfy the conditions of our theorems.

\section{Preliminaries} \label{prelimsec}

For a planar domain $U$, the function $(t,z)\mapsto\mathbf{P}_{z}(t<\tau_{U})$
satisfies the heat equation
\begin{equation}
\begin{cases}
\frac{1}{2}\Delta u-\frac{\partial u}{\partial t}=0\\
u(t,z)=0 & z\in\partial U
\end{cases}.\label{heat equation}
\end{equation}
The equation (\ref{heat equation}) is commonly referred to as being of \emph{Dirichlet
boundary condition} type\footnote{There are other categories, such as the Cauchy or Neumann types.}.
The {\it rate} of the solution of (\ref{heat equation}), which we denote
by $\lambda(U)$, is defined to be half of the principal Dirichlet
eigenvalue of $U$. This is the minimum of the spectrum of the
half of the Laplacian operator on $U$ combined with the boundary condition. Using
the expansion of the solution on the Hilbert basis generated by the
associated eigenfunctions, we can check that
\begin{equation}
-\frac{\ln\mathbf{P}_{z}(t<\tau_{U})}{t}\underset{t\rightarrow+\infty}{\longrightarrow}\lambda.\label{ln(P(t<T))}
\end{equation}

In \citep{mariano2020conformal}, the two authors treated the rate
of (\ref{heat equation}) on domains coming from the conformal
Skorokhod embedding. More precisely, they proposed the problem of finding, for a given distribution $\mu$, the $\mu$-domains that attain the highest
and lowest possible rate. That is, we seek two $\mu$- domains $U_{\mu}$
and $V_{\mu}$ such that
\[
\lambda(U_{\mu})\leq\lambda(\mathscr{D})\leq\lambda(V_{\mu})
\]
for all $\mu$-domains $\mathscr{D}$. As mentioned earlier, they partially solved this problem when $\mu$ is the uniform distribution on $(-1,1)$, showing that the catenary has the minimal rate among all $\mu$-domains. When we prove Theorem \ref{minimality of the domain1} we will see that this is a special case of a more general construction, one which always produces the minimal rate solution.

The analytic tools we will need, such as Fourier series, the periodic Hilbert transform, and the Hardy norm, are largely the same as used in \cite{gross2019} and \citep{boudabra2019remarks}.
For the sake of completeness,
we recall here two definitions as well as some related results.

A major tool for us is the periodic Hilbert transform.

\begin{defn}
The Hilbert transform of a $2\pi$-periodic function $f$ is defined
by
\[
H\{f\}(x):=PV\left\{ \frac{1}{2\pi}\int_{-\pi}^{\pi}f(x-t)\cot({\textstyle \frac{t}{2}})dt\right\} =\lim_{\eta\rightarrow0}\frac{1}{2\pi}\int_{\eta\leq|t|\leq\pi}f(x-t)\cot({\textstyle \frac{t}{2}})dt
\]
where $PV$ denotes the \emph{Cauchy principal value}.
\end{defn}

The Hilbert transform has some properties of great importance.
The Hilbert transform is an automorphism of $L^{p}$ and it satisfies
the \emph{strong type estimate}
\begin{equation}
||H_{f}||_{p}\leq\lambda_{p}||f||_{p}\label{strong inequality}
\end{equation}
for some positive constant $\lambda_{p}$ \citep[Vol I, page 203]{king2009hilbert}.  Furthermore, under suitable conditions it serves to swap the real
and the imaginary parts of the boundary values of holomorphic functions. That is, if $f$ is an analytic function on the disk which extends suitably to the boundary, then the Hilbert transform of $\Re(f)$ is $\Im(f)$ and the Hilbert transform of $\Im(f)$ is $\Re(f)$. Another important property is that the Hilbert operator $H$ commutes with positive dilations. That is, if $\Phi_{\lambda}\{f\}(x)=f(\lambda x)$ then
\begin{equation}
(H\circ\Phi_{\lambda})\{f\}=(\Phi_{\lambda}\circ H)\{f\}.\label{dilation}
\end{equation}



\begin{defn}
For any holomorphic function on the unit disc we define its $p^{th}$-
Hardy norm by
\begin{equation}
H_{p}(f):=\sup_{0\leq r<1}\left\{ \frac{1}{2\pi}\int_{0}^{2\pi}|f(re^{\theta i})|^{p}d\theta\right\} ^{\frac{1}{p}}.\label{hardy}
\end{equation}
The set of holomorphic functions whose $p^{th}$- Hardy norm is finite
is denoted by $\mathcal{H}^{p}$ and called the Hardy space (of index
$p$).
\end{defn}

The Hardy norm of a function $f$ is merely the
upper bound of $\{N_{r}(f)\}_{0<r<1}$ where
\[
N_{r}(f):=\left\{ \frac{1}{2\pi}\int_{0}^{2\pi}|f(re^{\theta i})|^{p}d\theta\right\} ^{\frac{1}{p}}\,\,\forall0<r<1.
\]
The quantity $N_{r}(f)$ is simply the $L^{p}$ norm of the restriction
$\theta\mapsto f(re^{\theta i})$. It can be shown, using harmonic
analysis techniques, that $N_{r}(f)$ is non-decreasing in terms of
$r$ \citep{Rudin2001}. This explains the use of $\sup$ in (\ref{hardy}).
Another crucial result about Hardy norms is that, if $H_{p}(f)$
is finite then $f$ has a radial extension to the boundary. More precisely
$f^{\ast}(z):=\lim_{r\rightarrow1}f(rz)$ exists for almost every $z\in\partial\mathbb{D}$ (with respect to Lebesgue measure), and this extension belongs to $L^{p}$ as well. In \citep{burkholder1977exit},
the author provided a powerful theorem that guarantees the equivalence
between the finiteness of the $p^{th}$ Hardy norm of $f$ and the
finiteness of $\mathbf{E}(\tau_{f(\mathbb{D})}^{\nicefrac{p}{2}})$.

Another important tool for us the following standard result.

\begin{lem}[Schwarz\textquoteright{} integral formula] \label{schwarzif}
 If $f\in\mathscr{H}^{p}$ then for all $z\in\mathbb{D}$
\begin{equation}
f(z)=\frac{1}{2\pi}\int_{0}^{2\pi}\frac{e^{ti}+z}{e^{ti}-z}\Re[f(e^{ti})]dt+i\Im[f(0)].\label{schwarz-1}
\end{equation}
\end{lem}

The formula (\ref{schwarz-1}) says that, under some assumptions,
that the boundary behavior of $\Re[f]$ determines entirely the map
$f$ inside $\mathbb{D}$. In particular, it implies Poisson integral
formula as $\Re[f(z)]$ is harmonic. Schwarz\textquoteright{} integral
formula is also used in the field of boundary value problem for analytic
functions. We refer the reader to \citep[Th. 17.26]{Rudin2001}, \cite[Ch. 7]{remmert2012theory}, or
\citep[Ch. I]{gakhov2014boundary} for more details.

\section{Proof of Theorem \ref{thm:The-distribution-}} \label{constructsec}

The proof builds on the methods used in \cite{gross2019}, but with some additional ideas required. Let $G$ be the quantile function for $\mu$, which is the pseudo inverse of the c.d.f of $\mu$
and consider the $2\pi$-periodic function
\[
\varphi_{\mu}:\begin{alignedat}{1}(0,2\pi) & \longrightarrow\mathbb{R}\\
\theta & \longmapsto G_{\mu}({\textstyle \frac{\theta}{2\pi}}).
\end{alignedat}
\]
It is a well known fact that $G_{\mu}(\mathsf{Uni}(0,1))$ has the distribution
$\mu$ \citep{devroye2006nonuniform}. Therefore, by scaling, if $\theta$
is uniformly distributed on $(0,2\pi)$ then
\begin{equation}
\varphi_{\mu}(\theta)\sim\mu.\label{mu}
\end{equation}
 As $\varphi_{\mu}\in L^{p}$ then it has a Fourier series whose partial
sum converge to it in $L^{p}$, i.e
\begin{equation}
\varphi_{\mu}(\theta)\overset{L^{p}}{=}\sum_{n=1}^{+\infty}(a_{n}\cos(n\theta)+b_{n}\sin(n\theta))\label{fourier}
\end{equation}
where $a_{n}$ and $b_{n}$ are the standard Fourier coefficients
\footnote{In general there is a constant $\frac{a_{0}}{2}$ added to the sum,
but it is omitted as it equals the average of $\mu$ which is assumed
zero.}. In fact, (\ref{fourier}) is also true in the almost everywhere
statement, which is the subject of the \emph{Carleson-Hunt theorem} \citep{Reyna2004,fefferman1973pointwise}.
The Hilbert transform of $\varphi_{\mu}$ is
\[
H\{\varphi_{\mu}\}(\theta)=\sum_{n=1}^{+\infty}(a_{n}\sin(n\theta)-b_{n}\cos(n\theta))
\]
and it belongs to $L^{p}$ as well \citep{butzer1971hilbert}. The
power series
\begin{equation} \label{define analytic}
\widetilde{\varphi}_{\mu}(z)=\sum_{n=1}^{+\infty}(a_{n}-b_{n}i)z^{n}
\end{equation}

belongs to $H^{p}$ since $\Re[\widetilde{\varphi}_{\mu}(e^{\theta i})]\overset{a.e}{=}\varphi_{\mu}(\theta)$
and $\Im[\widetilde{\varphi}_{\mu}(e^{\theta i})]\overset{a.e}{=}H\{\varphi_{\mu}\}(\theta)$.
The map $\widetilde{\varphi}_{\mu}(z)$ is one to one on the unit
disc $\mathbb{D}$ and maps $0$ to $0$. We give the proof of this
fact in a separate lemma. The domain $U:=\widetilde{\varphi}_{\mu}(\mathbb{D})$
is $\Delta^\infty$-convex since
$\varphi_{\mu}$ is non decreasing a.e on $[0,2\pi]$. Let $Z_{t}$
be a planar Brownian motion starting at $0$ and stopped at $\tau_{\mathbb{D}}$.
Then by conformal invariance $\widetilde{\varphi}_{\mu}(Z_{\tau_{\mathbb{D}}})$
is a planar Brownian motion starting at $\widetilde{\varphi}_{\mu}(0)=0$
and evaluated at $\tau_{U}$. As $Z_{\tau_{\mathbb{D}}}=e^{\theta i}$
where $\theta:=\mathrm{Arg}(Z_{\tau_{\mathbb{D}}})\sim\mathrm{\mathsf{Uni}}(0,2\pi)$,
then $\Re[\widetilde{\varphi}_{\mu}(Z_{\tau_{\mathbb{D}}})]=\varphi_{\mu}(\theta)$
has the distribution $\mu$ by (\ref{mu}). Since $0\in U$ and $\theta\mapsto\varphi_{\mu}(\theta)$
is non-decreasing it follows that it is $\Delta^\infty$-convex. Finally, the finiteness of $\mathbf{E}(\tau_{U}^{p})$
comes from theorem $4.1$ in \citep{burkholder1977exit}. \qed

\begin{lem}
The map $\widetilde{\varphi}_{\mu}(z)$ is one to one on the unit
disc.
\end{lem}

\begin{proof}
Recall that $G$ is a non-decreasing function on $(0,2\pi)$. We may find a sequence of functions $G_n$ on $(0,2\pi)$ which converges to $G$ in $L^1$ such that each $G_n$ has the following properties.

\begin{itemize}
    \item $G_n$ is $C^\infty$ and non-decreasing on $(0,2\pi)$.

    \item $\lim_{\theta \to 0+} G_n(\theta)$ and $\lim_{\theta \to 2 \pi -} G_n(\theta)$ both exist and are finite.

    \item $\lim_{\theta \to 0+} G^{(k)}_n(\theta)$ and $\lim_{\theta \to 2 \pi -} G_n^{(k)}(\theta)$ both exist and are finite, and furthermore  $\lim_{\theta \to 0+} G^{(k)}_n(\theta) = \lim_{\theta \to 2 \pi -} G_n^{(k)}(\theta)$,  for all $k \geq 1$. In other words, $G'_n$ extends to be a $C^\infty$ function on the circle.

\end{itemize}

If we now let $\kappa_n = \lim_{\theta \to 2 \pi -} G_n(\theta) - \lim_{\theta \to 0+} G_n(\theta)$, then $\tilde G_n(\theta) = G_n(\theta) - \frac{\kappa_n}{2 \pi} \theta$ extends to be $C^\infty$ on the entire circle (i.e. with $0$ and $2 \pi$ identified). A standard result in Fourier analysis now states that the Fourier coefficients $a_n$ of $\tilde G_n$ satisfy $a_n = O(|n|^{-2})$ (\cite[Cor. 2.4]{stein2011fourier}). Form an analytic function $\tilde f_n$ with the Fourier coefficients of $\tilde G_n(\theta)$ as in (\ref{define analytic}). The decay of the coefficients of $\tilde f_n$ means that the power series converges absolutely on the boundary of the disk, and $\tilde f_n$ therefore extends to be continuous on the closed unit disk. Let $f_n(z) = \tilde f_n(z) - i \frac{\kappa_n}{\pi} Log(1-z) + \frac{\kappa_n}{2}$, where $Log$ denotes the analytic logarithm function whose imaginary part $Arg$ takes values in $(-\pi,\pi)$. Then $f_n$ is analytic on the disc and continuous on the closure of the disc minus the point $1$. Furthermore, the modulus of $f_n$ approaches $\infty$ as $z$ approaches 1. It is therefore a continuous map from the closed unit disk to the Riemann sphere, with $1$ being mapped to $\infty$. Furthermore, it can be checked by elementary geometry, or by trigonometric identities, that

$$
\Re(- i \frac{\kappa_n}{\pi} Log(1-e^{i \theta})) = \frac{\kappa_n}{\pi} Arg(1-e^{i \theta}) = \frac{\kappa_n}{\pi}(-\frac{\pi}{2} + \frac{\theta}{2}) = \frac{-\kappa_n}{2} + \frac{\kappa_n}{2\pi}\theta.
$$

It follows that $\Re(f_n(e^{i\theta})) = \tilde G_n(\theta) + \frac{-\kappa_n}{2} + \frac{\kappa_n}{2\pi}\theta + \frac{\kappa_n}{2} = G_n(\theta)$. $\Re(f_n(e^{i\theta}))$ is therefore strictly increasing on $(0,2\pi)$; that is, as $\theta$ increases from $0$ to $2\pi$, the image $f_n(e^{i\theta})$ "winds once" about the domain in the Riemann sphere. This proves that $f_n$ is injective; see for instance \cite[Thm. 10.31]{Rudin2001}.

Using Lemma \ref{schwarzif}, it is straightforward to verify that the $L^1$ convergence of $G_n$ to $G$ on $(0,2\pi)$ implies that $f_n$ converges to $\tilde \phi_\mu$ uniformly on compact sets (note that $\Im(f_n(0)) = Im(\tilde \phi_\mu(0)) =0$ by construction). By Hurwitz's Theorem (see \cite{remmert2012theory}), $\tilde \phi_\mu$ is either constant or injective. The case when $\tilde \phi_\mu$ is constant corresponds to the case when $\mu$ is a point mass at the origin, and since we have excluded this case we see that $\tilde \phi_\mu$ is injective.
\end{proof}

{\bf Remark:} The primary difference between this proof and that in \cite{gross2019} is essentially that the $G_n$'s have a jump discontinuity at the point $0 (\equiv 2\pi)$ when viewed as a function defined on the circle. This is why the logarithm made an appearance.

{\bf Remark:} The assumption that $\mu$ is nondegenerate appears when applying Hurwitz's Theorem, and this same issue also applies to the proof given in \cite{gross2019}. Essentially, in this case the domain would degenerate to a point at the origin; it could not for instance be the domain limited by the boundary $\{Re(z) = 1, |Im(z)| \} \geq 1$, since this domain contains a half-plane and therefore $E[(\tau_U)^p] = \infty$

\section{A uniqueness criterion for $\mu$-domains} \label{uniquesec}

Now we are ready to tackle the uniqueness issue of our
$\Delta^\infty$-convex domain. Before that, we need some preliminary
tools.

\begin{defn}
\label{a.e non decrezasing} We say that a function $F$ defined $[a,b]$
is non-decreasing almost everywhere if there is a non-decreasing function $\hat F$ defined on all of $[a,b]$ such that $F=\hat F$ almost everywhere. For such a function we define its generalised inverse
function $F^{-1}$ by
\[
F^{-1}(x)=\sup\{t\in[a,b]\mid \hat F(t)\leq x\}=\inf\{t\in[a,b]\mid \hat F(t)>x\}
\]
with the convention $F^{-1}(x)=a$ if $\{t\in[a,b]\mid F(t)\leq x\}$
is empty. The swap between $\sup$ and $\inf$ is justified by the
monotonicity of $\hat F$.
\end{defn}

\begin{lem}
\label{equality}Let $F$ and $G$ be two function defined on $[a,b]$,
non-decreasing a.e. If $\widetilde{F}(x)=\widetilde{G}(x)$ for a.e.
$x$ then $F$ and $G$ agree a.e.
\end{lem}

\begin{proof}
Let $\Lambda_{F}$ and $\Lambda_{G}$ be the subsets of $[a,b]$ of full measure upon which $F=\hat F$ and $G=\hat G$, respectively. Set $\Lambda=\Lambda_{F}\cap\Lambda_{G}$ and suppose that $\hat F(c)\neq \hat G(c)$
for some $c\in\Lambda$, say $F\hat (c)<\hat G(c)$. Since the subset
of $\Lambda$ where $\hat F$ or $\hat G$ are discontinuous is countable (\citep{Rudin1976}) and therefore of measure 0, we may discard it and assume that
$c$ is a continuity point for both of them in $\Lambda$. Choose $\ell_1, \ell_2$ such that
$\hat F(c) < \ell_1 < \ell_2 < \hat G(c)$. The continuity of $\hat F$ and $\hat G$ at $c$
yields
\[
F(t)<\ell_1 < \ell_2 <G(t)
\]
for all $t\in[c-\delta,c+\delta]\cap\Lambda$ for some $\delta>0$. Now, if $y \in (\ell_1,\ell_2)$ then $G^{-1}(y) \leq c-\delta$ but $F^{-1}(y) \geq c+\delta$. Thus, $F^{-1}$ and $G^{-1}$ disagree on a set of positive measure. This proves the lemma.
\end{proof}

Now we state
the uniqueness result.
\begin{thm}
\label{uniqueness}The distribution $\mu$ generates a unique $\mu$-domain
which is $\Delta^\infty$-convex provided that $\mathbf{E}(\tau_{U}^{p})<\infty$
for some $p>1$.
\end{thm}

\begin{proof}
Let $U,V$ be two $\mu$-domains which are $\Delta^\infty$-convex and which satisfy $\mathbf{E}(\tau_{U}^{p}),\mathbf{E}(\tau_{V}^{p})<\infty$ for some $p>1$. Let $f,g$ be two conformal
maps from $\mathbb{D}$ to $U,V$ fixing $0$. By $\Delta^\infty$-convexity, the functions $F(t)=\Re(f(e^{ti}))$
and $G(t)=\Re(g(e^{ti}))$ (defined for a.e $t\in[0,2\pi]$ in the
sense of radial limits, see \cite{Rudin2001}) are a.e. non-decreasing so they have well defined
generalized inverse functions $\widetilde{F}$ and $\widetilde{G}$.
As $U$ and $V$ are $\mu$-domain then $\Re[Z_{\tau_{U}}]$ and $\Re[Z_{\tau_{V}}]$
share the same distribution $\mu$. Therefore, by the conformal invariance
of Brownian motion, we get
\[
\begin{alignedat}{1}\mathbf{P}\{\Re[Z_{\tau_{U}}]\in(-\infty,x]\} & =\mathbf{P}\{\Re[Z_{\tau_{V}}]\in(-\infty,x]\}\\
 & =\frac{\widetilde{F}(x)}{2\pi}\\
 & =\frac{\widetilde{G}(x)}{2\pi}.
\end{alignedat}
\]
for all $x$. Thus $F\overset{a.e}{=}G$ by applying Lemma \ref{equality}.
Consequently $f-g$ is constant via Schwarz integral formula (\ref{schwarz-1}),
and since they send zero to the same point we conclude the equality
of $f$ and $g$.
\end{proof}

\begin{rem}
The condition that $\mathbf{E}(\tau_{U}^{p})<\infty$
for some $p>1$ is necessary. To see this, take for instance the domain $U = \CC \backslash  \{\Im(z) \leq -1, -1 \leq \Re(z) \leq 1\}$. The resulting distribution of $\Re[Z_{\tau_{U}}]$ is supported on $[-1,1]$, and therefore has all moments. However, $U$ cannot be the domain generated in Theorem \ref{Existence1} as $\mathbf{E}(\tau_{U}^{p})=\infty$ for $p \geq \frac{1}{4}$.
\end{rem}

Theorem \ref{uniqueness} implies that, in practice, to obtain such
$\mu$-domain it it enough to find the Fourier expansion of
\[
\varphi_{\mu}:\begin{alignedat}{1}(0,2\pi) & \longrightarrow\mathbb{R}\\
\theta & \longmapsto G_{\mu}({\textstyle \frac{\theta}{2\pi}})
\end{alignedat}
\]
and consider $\widetilde{\varphi}_{\mu}(\mathbb{D})$.
\begin{thm}
\label{minimality of the domain}The $\mu$ domain $U_{\mu}$ constructed
in $\mathbf{E}(\tau_{U}^{p})<\infty$ has the lowest rate among all other $\mu$-
domains. In other words
\[
\lambda(\widehat{U}_{\mu})\leq\lambda(U_{\mu})
\]
 for all $\mu$ domains $U_{\mu}$. Furthermore $\lambda(\widehat{U}_{\mu})=\frac{\pi^{2}}{2(\beta-\alpha)^{2}}$
where $[\alpha,\beta]$ is the smallest interval containing the support
of $\mu$.
\end{thm}

The proof of Theorem \ref{minimality of the domain} is exactly the same as for Theorems
$2$ and $3$ in \citep{mariano2020conformal}. It is based on the rate of
strips and rectangles as well as the monotonicity of the rate in terms
of domains.

\begin{example}
The uniform distribution on $(-1,1)$.
\end{example}

We can check that $\varphi_{\mu}(\theta)=\frac{\theta}{\pi}-1$ and
has the Fourier series
\[
\varphi_{\mu}(\theta)=-\frac{2}{\pi}\sum_{n=1}^{+\infty}{\textstyle \frac{\sin(n\theta)}{n}}.
\]
The power series is
\[
\widetilde{\varphi}_{\mu}(z)=\frac{2i}{\pi}\sum_{n=1}^{+\infty}\frac{z^{n}}{n}=-\frac{2i}{\pi}\ln(1-z).
\]
This function maps the unit disk to the catenary; see Figure \ref{cat12}. This example is the subject of \cite{mariano2020conformal}. Note that $-\frac{2i}{\pi}\ln(1+z)$ produces the same distribution,
however $-\frac{2i}{\pi}\ln(1-z)$ is the one which has a non decreasing
real part on $(0,2\pi)$.

\begin{figure}
\begin{centering}
\includegraphics[width=6cm,height=6cm,keepaspectratio]{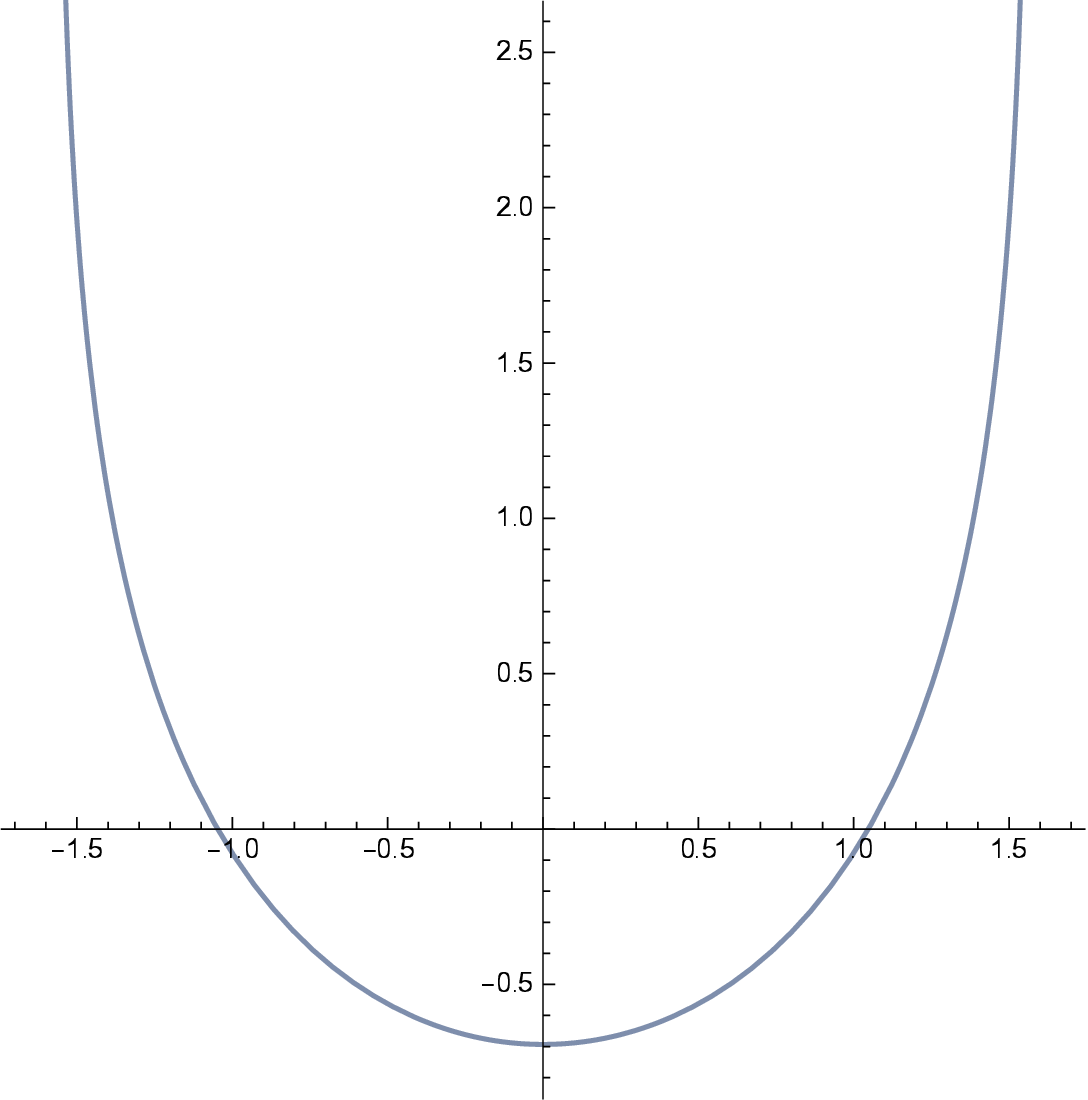}
\par\end{centering}
\caption{The catenary, which is domain for the uniform distribution.} \label{cat12}
\end{figure}

\begin{example}
\label{exa:The-scaled-and}The scaled and centered arcsine law on
$(-1,1)$.
\end{example}

We get $\varphi_{\mu}(\theta)=-\cos(\nicefrac{\theta}{2})$ and so
the power series is
\[
\widetilde{\varphi}_{\mu}(z)=-\frac{8i}{\pi}\sum_{n=1}^{+\infty}{\textstyle \frac{n}{1-4n^{2}}z^{n}}={\textstyle \frac{i}{\pi}\left\{ \ln\left(\frac{1+\sqrt{z}}{1-\sqrt{z}}\right)(\sqrt{z}+\frac{1}{\sqrt{z}})-2\right\} }.
\]

It is perhaps not so easy to deduce the image of the unit disk under this map; however, in the next section we present a method for finding the equation of the boundary curve in terms of the distribution. As we will show there, when applied to this law we obtain the domain in Figure \ref{cat11}, limited by the curve $y=-\frac{2}{\pi}(x\ln(\cot({\textstyle \frac{\arccos(-x)}{4}}))+1)$.

\begin{figure}
\begin{centering}
\includegraphics[width=6cm,height=6cm,keepaspectratio]{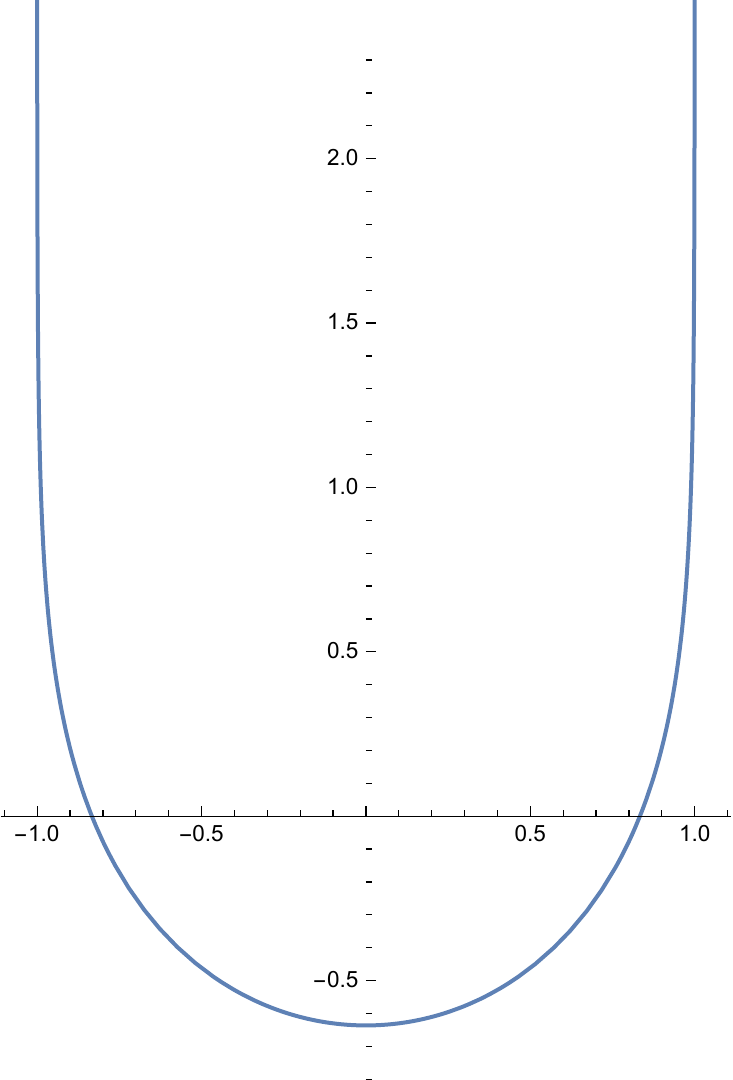}
\par\end{centering}
\caption{The domain for the arcsin distribution. The extremal lower point of the domain is $-\frac{8i}{\pi}\sum_{n=1}^{+\infty}{\textstyle \frac{(-1)^{n}n}{1-4n^{2}}}=-\nicefrac{2i}{\pi}\approx-0.636i$} \label{cat11}

\end{figure}

\begin{example}
Consider the density $\frac{\mathrm{sech}(\frac{\sqrt{2}\pi}{2}x)}{\sqrt{2}}$. As is shown in \cite{boudabra2019remarks}, Gross' method applied to this distribution yields a horizontal infinite strip. 
The method given in Theorem \ref{Existence1} when applied to this distribution yields the function

\[
f(z)=\frac{2i}{\pi^{2}}\Big(\ln\left(\frac{1+\sqrt{z}}{1-\sqrt{z}}\right)\Big)^{2}.
\]

This conformal map sends the unit disc onto the parabola limited by the equation \citep{kanas1999conic}

\begin{equation}
2y=x^{2}-1.\label{u,v}
\end{equation}
\end{example}

\begin{figure}[H]
\begin{centering}
\includegraphics[width=5cm,height=5cm]{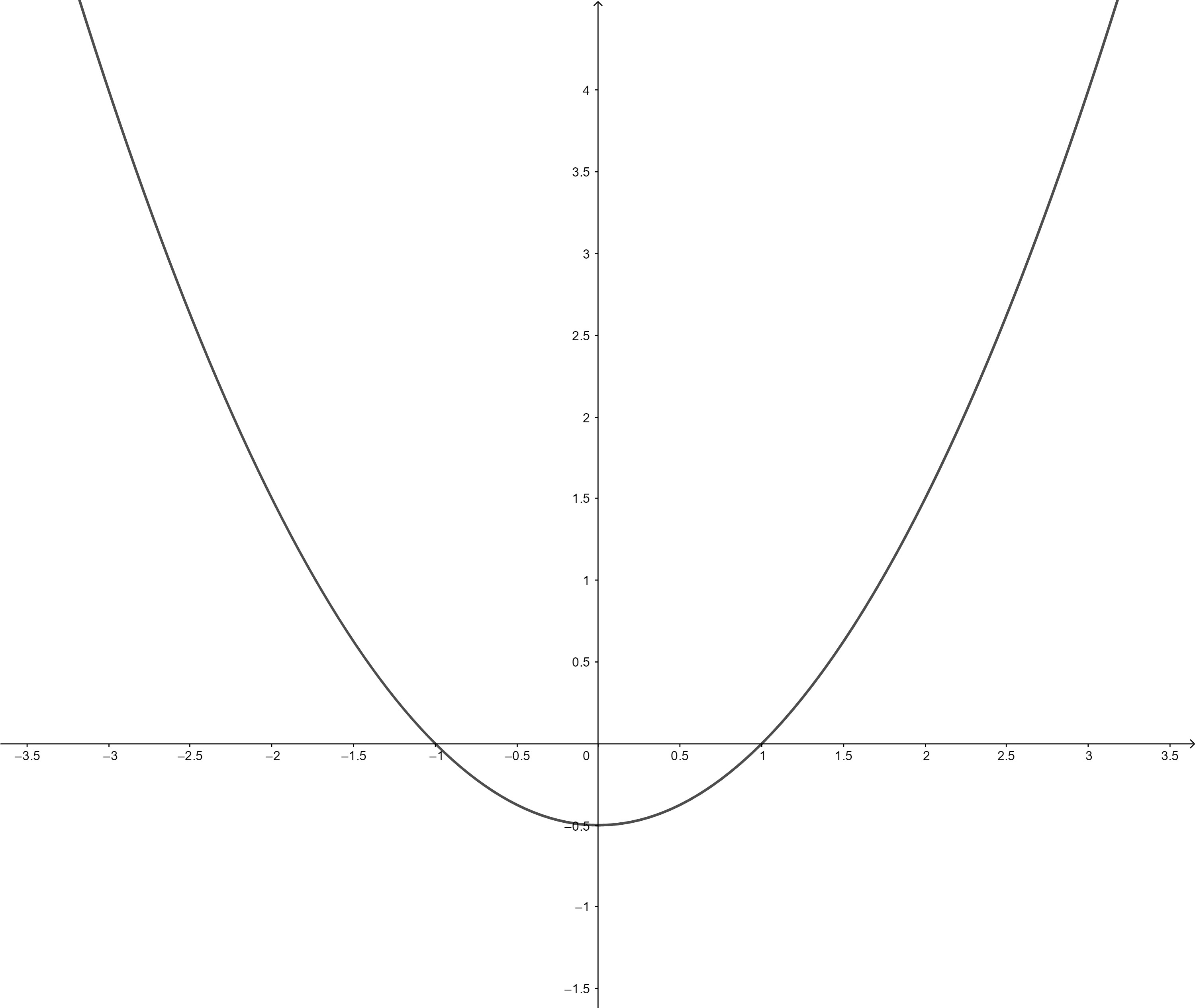}
\par\end{centering}
\caption{Parabola of equation $x^{2}=2y+1$.}
\end{figure}

\begin{example}
If the distribution $\mu$ is of the form
\[
\mu=\sum_{n=1}^{m}x_{n}\delta_{x_{n}}
\]
 where $x_{n}>0$ for all $n$ and $\sum_{n-1}^m x_n = 1$, then the domain generated
by $\mu$ is the strip $\{x_{1}<\Re(z)<x_{m}\}$ with the vertical
slits $\left(\{x_{n}\}\times(-\infty,y_{n}]\right)_{1<n<m}$ removed, where
the $y_{n}$'s are some real numbers. In \cite{gross2019} some methods for calculating the values of $y_n$ in the context of Gross' method are presented, and they also work when applied to our method.
\end{example}



\section{Equation of the boundary} \label{boundarysec}

In this section we prove that, in some situations, the boundary of
the $\mu$-domain is the graph of a function. This often helps to
reduce the computations required to determine the domain. As before,
$F_{\mu}$ stands for the c.d.f of $\mu$ and $G_{\mu}$ stands for
its pseudo-inverse.

\begin{thm} \label{boundaryequation}
If the distribution $\mu$ is atomless then the boundary of our domain
have the equation
\begin{equation}
y=\gamma(x)=H\{\varphi_{\mu}\}(\varphi_{\mu}^{-1}(x))=H\{F_{\mu}^{-1}\}(F_{\mu}(x)).\label{curve}
\end{equation}
\end{thm}

\begin{proof}
Let $U$ be such a domain. Due to the atomlessness of $\mu$ and the $\Delta^\infty$-convexity of $U$ axis then any vertical
line crosses $\partial U$ in at most one point. Hence this boundary
is the curve of a function $x\longmapsto\gamma(x)$. The uniqueness
of $U$ guaranteed by Theorem \ref{uniqueness} yields that $\Re[f(e^{\theta i})]$
is simply $\varphi_{\mu}(\theta)=G_{\mu}(\nicefrac{\theta}{2\pi})$.
As $G_{\mu}$ is continuous then $\gamma$ will inherits this property.
The boundary of $U$ is parameterized by
\[
\theta\mapsto(x,y)=(\varphi_{\mu}(\theta),H\{\varphi_{\mu}\}(\theta)).
\]
Therefore to find $\gamma$ it is enough to express $y$ in terms
of $x$. Since $\varphi_{\mu}$ is increasing then $\theta=\varphi_{\mu}^{-1}(x)$
and hence
\[
\begin{alignedat}{1}\gamma(x) & =H\{\varphi_{\mu}\}(\varphi_{\mu}^{-1}(x))\\
 & \overset{{\scriptscriptstyle \varphi_{\mu}^{-1}=2\pi F_{\mu}}}{=}H\{\Phi_{\nicefrac{1}{2\pi}}\{F_{\mu}^{-1}\}\}(2\pi F_{\mu}(x))\\
 & \overset{\eqref{dilation}}{=}H\{F_{\mu}^{-1}\}(F_{\mu}(x)).
\end{alignedat}
\]
\end{proof}
The above theorem is also valid for domains obtained by Gross method
where $y=\gamma(x)$ is the equation of the lower boundary. Let's
give two concrete examples where (\ref{curve}) is applied, one for
our method and one for Gross method. Furthermore, the boundary equation
(\ref{curve}) is a local formula and we can use it ``carefully''
for atomic distributions. More precisely, if $\mu$ has atoms $x_{1},x_{2},...$
and $F_{\mu}$ is increasing on every interval $(x_{i},x_{i+1})$
then (\ref{curve}) is valid for all $x\in\textrm{supp}(\mu)\setminus\{x_{1},...\}$.

\begin{prop}
Let $\gamma(x)$ be a continuous and a.e differentiable function defined
over some (finite or infinite) interval $(a,b)$ such that $\gamma(0)<0$.
If
\[
\gamma(x)=H\{F^{-1}\}(F(x))
\]
for some continuous distribution function $F$ then the density of
$Z_{\tau_{U}}$ at $z=x+yi$ is given a.e by
\[
\rho(x+yi)=\frac{F'(x)}{\sqrt{1+\gamma'(x)^{2}}}
\]
where $U$ is the domain above the graph of $\gamma$.
\end{prop}

\begin{proof}
The proof comes from the formula provided in \citep{boudabra2019remarks}.
\end{proof}
\begin{rem}
Suppose $\mu$ has no atoms and $U$ is the $\mu$-domain generated by Gross' method. If $\gamma(x)$ denotes the function determining the lower boundary of $U$, then the same argument shows
\[
\rho(x\pm\begin{vmatrix}y\end{vmatrix}i)=\frac{F'(x)}{2\sqrt{1+\gamma'(x)^{2}}}.
\]
\end{rem}

\begin{example}
In \citep[Thm.3]{mariano2020conformal}, the authors gave the following
domain
\[
\mathbb{U}:=\{(x,y)\mid-1<x<1,\,y>-{\textstyle \frac{2}{\pi}}\ln(2\cos(\nicefrac{\pi x}{2}))\}
\]
as an example of a $\mu$-domain where $\mu=\mathsf{Uni}(-1,1)$.
The domain $\mathbb{U}$ is $\Delta^\infty$-convex so it is unique by our Theorem \ref{uniqueness}.
We show now that the function $\phi:x\mapsto-{\textstyle \frac{2}{\pi}}\ln(2\cos(\nicefrac{\pi x}{2}))$ can be deduced from Theorem \ref{boundaryequation} as expected. The uniform distribution c.d.f and its pseudo-inverse
function are given by
\[
\begin{alignedat}{1}F_{\mu}(x) & ={\textstyle \frac{x+1}{2}}\\
F_{\mu}^{-1}(u) & =2u-1.
\end{alignedat}
\]
The (periodic) Hilbert transform of $F_{\mu}^{-1}$ is
\[
H\{F_{\mu}^{-1}\}(x)=-{\textstyle \frac{2}{\pi}}\ln(2\sin(\pi x)),
\]
and therefore we get
\[
\begin{alignedat}{1}\gamma(x) & =-{\textstyle \frac{2}{\pi}}\ln(2\sin(\pi({\textstyle \frac{x+1}{2}})))\\
 & =-{\textstyle \frac{2}{\pi}}\ln(2\cos(\nicefrac{\pi x}{2}))\\
 & =\phi(x)
\end{alignedat}
\]
\end{example}

\begin{example}
We mention that the Gross domain generated by the centered and
scaled arcsine distribution mentioned in Example \ref{exa:The-scaled-and}
is simply the unit disc. That is, after performing the necessary
computations and applying Theorem \ref{boundaryequation}, we find the equation of the lower boundary
\[
\begin{alignedat}{1}\gamma(x) & =-\sin(\arcsin(x)+\nicefrac{\pi}{2})\\
 & =-\sqrt{1-x^{2}}.
\end{alignedat}
\]
Consequently, the generated Gross domain is limited by the union of
the graphs of $x\mapsto\pm\sqrt{1-x^{2}}$. This is the unit disc.
The same technique shows that, for the same distribution, the boundary
equation of our domain is
\[
\gamma(x)=-\frac{2}{\pi}(x\ln(\cot({\textstyle \frac{\arccos(-x)}{4}}))+1).
\]
\end{example}

\section{A suprising pseudo-example} \label{cauchy}

The density of the standard Cauchy distribution $\mu$ is given by
\[
\varrho_{\mu}(x)=\frac{1}{\pi(1+x^{2})}
\]
and its quantile function is $G_{\mu}(u)=-\cot(\pi u)$ for all $u\in(0,1)$.
This distribution does not have a mean as it is not
integrable. Therefore, we can not apply the same techniques as before
to generate a corresponding $\mu$-domain. However, let us simply ignore this issue and apply the method formally. It can be shown that $\Re[Z_{\tau_{\mathscr{U}}}]$
has the density $\varrho_{\mu}$ where $\mathscr{U}$ is the upper
half plane limited by $\{z\mid\Im(z)=-1\}$; a recent proof of this using the optional stopping theorem appears in \citep{chinjungmark}, but it can also be deduced by a direct calculation, using the Poisson kernel, or by properties of stable distributions; see for instance \cite[Sec. 1.9]{dur} or \cite[Ch. VI.2]{feller2008introduction}.
We can check that
\begin{equation} \label{cotident}
{\textstyle 2i\frac{e^{\theta i}}{1-e^{\theta i}}}=-\cot({\textstyle \frac{\theta}{2}})-i=G_{\mu}({\textstyle \frac{\theta}{2\pi}})-i.
\end{equation}
Note that the Fourier coefficients technically do not exist here, because the function is not in $L^1$ (however, the sine integrals against $-\cot({\textstyle \frac{\theta}{2}})$ do converge, and if we set the cosine terms all to $0$ by the oddness of $-\cot({\textstyle \frac{\theta}{2}})$ then we obtain (\ref{cotident})). Nevertheless, we have
\[
\vartheta(z)=2i\sum_{n=1}^{\infty}z^{n}=2i\frac{z}{1-z}.
\]
This is the M\"obius transformation taking the disk to the half-plane $\{z\mid\Im(z)>-1\}$, so that the correct conclusion does hold in this case. We do not know whether this is an isolated coincidence or a sign that the method is extendable. It is interesting to note that the same type of formal calculations do not seem to apply to Gross' method.

\section{Acknowledgements}

The authors would like to thank Zihua Guo, Renan Gross, Phanuel Mariano, and Hugo Panzo for helpful communications.

\bibliographystyle{plain}
\bibliography{Maheref}

\end{document}